\documentclass[reqno]{amsart}
\usepackage{graphicx, enumerate,amssymb}
\newtheorem{thm}{Theorem}[section]

\newtheorem{lem}[thm]{Lemma}
\newtheorem{prop}[thm]{Proposition}
\newtheorem{result}[thm]{Result}
\theoremstyle{definition}

\theoremstyle{remark}

\numberwithin{equation}{section}
\newcommand{\norm}[1]{\left\Vert#1\right\Vert}
\newcommand{\abs}[1]{\left\vert#1\right\vert}
\newcommand{\rl}{{\mathbb{R}}}
\newcommand{\cx}{{\mathbb{C}}}
\newcommand{\dist}{{\mathrm { dist}}}
\newcommand{\tensor}{\otimes}
\newcommand{\csor}{\widehat{\otimes}}

\newcommand{\psh}{plurisubharmonic }

\newcommand{\D}{\mathbb{D}}
\newcommand{\dbar}{\overline{\partial}}

\newcommand{\tmop}[1]{\ensuremath{\operatorname{#1}}}
\renewcommand{\Re}{\tmop{Re}}

\newcommand{\pss}{\widetilde{W}}

\title[Condition R and proper holomorphic  maps between product domains]{Condition R and proper holomorphic  maps between equidimensional  product domains}
\thanks{Kaushal Verma was supported in part by the DST SwarnaJayanti fellowship 2009--2010 and the UGC--CAS Grant}
\author{Debraj Chakrabarti}
\address{Department of Mathematics, Central Michigan University, Mt. Pleasant,  MI 48859,  USA}
\email{chakr2d@cmich.edu}
\author{Kaushal Verma}
\address{Department of Mathematics, Indian Institute of Science, Bengaluru-560012, India}
\email{kverma@math.iisc.ernet.in}
\begin{document}
\begin{abstract}We consider proper holomorphic mappings of equidimensional pseudoconvex domains in complex Euclidean space,
where both source and target can be represented as Cartesian products of smoothly bounded domains. It is shown that such 
mappings extend smoothly up to the closures of the domains, provided each factor of the source satisfies Condition R. 
It also shown that the number of smoothly bounded factors in the source and target must be the same, and the proper
holomorphic map splits as product of proper  mappings between  the factor domains.
\end{abstract}

\maketitle
\section{Introduction}

\subsection{Proper mapping of product domains} Let $D$ and $G$ be bounded domains in $\cx^n$, each of which can be represented as a product
of smoothly bounded domains, where `smooth' in this article always means $\mathcal{C}^\infty$. More precisely,
there exist positive integers $k$ and $l$, so that 
$D = D^1 \times D^2 \times \ldots \times D^k$ and
$G = G^1 \times G^2 \times \ldots \times G^l$ where  for $1 \le i \le 
k$ ,  there is a positive integer $\mu_i$ such that  $D^i \subset \mathbb C^{\mu_i}$ is a smoothly bounded domain,
and similarly for  $1 \le j \le l$, there is  a postive integer $\nu_j$ such that   
$G^j \subset \mathbb C^{\nu_j}$ is a  smoothly bounded domain. 
Of course then we  have $\mu_1 + \mu_2 + \ldots + \mu_k = \nu_1 + \nu_2 + \ldots + \nu_l = n$.
Recall that the {\em Bergman Projection} on a domain $\Omega$ in complex Euclidean space is 
the orthogonal projection 
from the Hilbert space  $L^2(\Omega)$ of  functions square-integrable with respect to the standard Lebesgue measure
 to the  closed subspace $\mathcal{H}(\Omega)$ of square-integrable 
holomorphic functions. Recall further, that the domain $\Omega$ is said to satisfy {\em Condition R} if
the Bergman projection maps the space $\mathcal{C}^\infty(\overline{\Omega})$ of functions smooth up to the boundary on $\Omega$,
to the space $\mathcal{H}^\infty(\Omega)=\mathcal{O}(\Omega)\cap \mathcal{C}^\infty(\overline{\Omega})$ of holomorphic functions smooth up to the boundary.
Our main result is:

\begin{thm}\label{thm-main}
Assume that each of the factor domains $D^i$ constituting the product $D$ is pseudoconvex and  satisfies Condition R. Let 
$f : D \rightarrow G$ be a proper holomorphic mapping. 
Then,
\begin{enumerate}[(i)]
  \item $f$ extends  to a $\mathcal{C}^\infty$ map from  $\overline D$ to $\overline{G}$. 
  \item  $l = k$, i.e., the number of smooth factors in the domain and range coincide,  and 
  \item  there is a permutation $\sigma$ of the set $\{1,\dots, k\}$, such that  for   $1\leq i\leq k$,
  we have $\mu_i=\nu_{\sigma(i)}$, and there is a proper holomorphic map
  $f^i : D^i \rightarrow G^{\sigma(i)}$ such that    $f:D\to G$ is represented as a product mapping
    \[ f =  f^1 \times f^2 \times \ldots \times f^k\]. 
  \end{enumerate}
\end{thm}

\noindent {\em Remark: } The hypotheses of the theorem imply that $G$, being the image of a pseudoconvex domain under a proper 
holomorphic mapping, is also pseudoconvex. Thus each factor $G^j$ must be pseudoconvex.

\subsection{Poincar\'{e} meets Fefferman} In the theory of holomorphic mappings in several variables, there 
are two well-known types of results. The simplest result of the first type is traditionally attributed to Poincar\'{e},
and states that the unit ball in $\cx^2$ cannot be biholomorphically mapped onto the unit bidisc. This has been generalized
in many directions (see \cite{RS,R,T,Huck1} etc.)  A result of Rischel \cite{R} implies that no strongly pseudoconvex domain in $\cx^n$ can be properly mapped
onto a product domain. In another direction, for biholomorphic maps,  a result of Tsyganov \cite{T} states that any biholomorphic
map of smoothly bounded product domains is represented as a product of biholomorphic maps in the factors.

\medskip

Another class of results regarding holomorphic maps generalizes  the boundary regularity of conformal mappings of smooth domains 
in one variable. A famous result of Fefferman \cite{fef} states that a biholomorphic map of strongly pseudoconvex domains smoothly 
extends to the closures of the domains. This has been generalized to the following equally famous theorem:
\begin{result}[{\cite{BL, B2,Be1,Be2,BC,DF}}]\label{res-condnr} A proper mapping of smoothly bounded equidimensional
pseudoconvex domains in complex Euclidean
space  extends smoothly up to the boundary, provided the source domain of the map satisfies condition R.
\end{result}
Indeed it is conjectured that such smooth extension to the 
boundary actually holds for proper mappings between arbitrary smoothly bounded pseudoconvex domains, 
and the hypothesis of condition R is redundant.

\medskip

Theorem~\ref{thm-main}, the main result of this paper, shows the close relation between the `Poincar\'{e} type' and 
`Fefferman type' results mentioned above. Indeed, the novelty of our approach here is that 
the proof of Theorem~\ref{thm-main} is an adaptation of  the classical  technique used to prove 
Result~\ref{res-condnr}.
The key observation, explicitly made in \cite{CS} is that Condition R 
holds on a product domain provided it holds on each factor domain. This is a direct consequence of  the formula representing 
the Bergman kernel of a product domain as the product of the Bergman kernels of the factors.

   The  arguments of Bell-Catlin-Diederich-Fornaess-Ligocka, suitably modified, can be applied to product domains, in order to conclude that a proper mapping of product domains extends
 as a continuous mapping of the closures of the domains. At this point, one can use a method of Ligocka (see \cite{L}) to complete the argument
 to show that the proper map actually splits. The smooth extension to boundary follows from the Bell-Catlin  result applied to each factor.

\subsection{General Remarks}We note that it is possible to generalize Theorem~\ref{thm-main} to the situation of a  product of smoothly
bounded domains in Stein manifolds. The proof given here goes through, when rewritten in invariant language.

It is well-known that results of either the ``Poincar\'{e} type" or of the ``Fefferman type"  are not statements only about the 
complex structures of the domains being mapped, but depend crucially on the Hermitian metric up to the boundary.
For example, we cannot replace the ambient manifolds in Poincar\'{e} type results by arbitrary manifolds, as shown in \cite{ohsawa}, where 
the example of a smoothly bounded domain in a compact manifold is given, which is biholomorphic to a product domain.
The fact that Fefferman-type results do not generalize to domains in compact manifolds follows from the closely related
construction in \cite{Ba0}. 
Also of interest in this connection is the example in \cite{fri} of two  domain  in $\cx^2$, each biholomorphic to the bidisc, 
but such that any biholomorphism between them does not extend continuously to the closures. 
This shows that Carath\'{e}odory's theorem
on conformal mapping has no general analog in several variables, and more importantly we do need actual product 
domains for the result of Theorem~\ref{thm-main} to hold (as distinguished from domains in $\cx^n$ 
biholomorphic  to product domains.) Results related to Theorem~\ref{thm-main} can be found in \cite{L,N,R, T} and \cite{Z}.

There is good reason to speculate whether the hypothesis of Condition R  on the factors of the source domain
is really necessary for the truth of the conclusion of Theorem~\ref{thm-main} to hold. However, our goal here 
is to employ techniques used in the theory of extension of mappings to the boundary,  in the proof of a  Poincar\'{e} type
result for proper mappings as far as possible, and  the application of these techniques does require this hypothesis. 

We also note that many of the techniques of this paper generalize to the class of domains referred in \cite{Ba} as 
``domains with non-degenerate corners."  In a future work, we will consider the holomorphic mappings of this class of
domains in full generality, and here restrict ourselves to the product domains.

In the following Section~\ref{sec-functions} we recall some constructions of function spaces on product domains.
Most of the results here can be found in \cite{prod}. Next, in Section~\ref{sec-bergman}, we generalize some well-known 
classical facts regarding spaces of holomorphic function to product domains. Sections~\ref{sec-estimates} and 
\ref{sec-proof} complete the proof of  Theorem~\ref{thm-main}.  For completeness, we include proofs of several 
intermediate statements and lemmas which are well-known for smooth domains.

\subsection{Acknowledgements} We thank  Mei-Chi Shaw for bringing this question to our attention.  We also
thank   David Barrett  and  Dror Varolin for interesting  discussions and comments. Debraj Chakrabarti thanks
the University of Western Ontario for its hospitality where part of this work was done. We also thank the referee for  helpful comments 
on the first version of this paper.


\section{Function spaces  on Product Domains}\label{sec-functions}
\subsection{Smooth functions}
Let $s$ be  a positive integer or $\infty$.  We recall that a function on a (not necessarily open) subset $E$ of $\rl^N$ 
is said to be of class $\mathcal{C}^s$
(in the sense of Whitney)
if  there is an open neighborhood of $E$ to which the function can be extended as a $\mathcal{C}^s$ function in the usual sense.
If $E$ is the closure $\overline{\Omega}$  of a  Lipschitz Domain $\Omega$ (i.e. a domain in which after an affine change of coordinates
 the  boundary can be locally represented as the graph of a Lipschitz function)
this definition coincides with any other reasonable definition of a $\mathcal{C}^s$ function, and further,
a function is in $\mathcal{C}^\infty(\overline{\Omega})$ if it is in $\mathcal{C}^s(\overline{\Omega})$ 
for each positive integer $s$  (see \cite[Chapter VI]{stein}.)

\medskip

On a domain $\Omega\subset \rl^N$ in Euclidean space, 
the  ($L^2$-)Sobolev space $W^s(\Omega)$  of positive integral order $s$ is the space  of those distributions
on $\Omega$, which are in $L^2$ along with all distributional partial derivatives of order up to  $s$. Recall that such a space is a Hilbert space 
with the usual inner product. We will be interested in the case when $\Omega$ is
Lipschitz.  A fundamental
fact regarding the space $W^s(\Omega)$, when $\Omega$ is Lipschitz is the following (see \cite[Page~181, Theorem 5$^\prime$]{stein}):
{\em There exists a continuous linear extension operator  $\mathfrak{E}$ from $W^s(\Omega)$ to $W^s(\rl^N)$: i.e., for each $f\in W^s(\Omega)$, we have
$(\mathfrak{E}f)|_\Omega=f$.} This, along with Sobolev embedding of $W^s(\rl^N)$ for large $s$ in 
a space of smooth functions, imply that on a domain $\Omega$ with Lipschitz boundary, we have
\[ \mathcal{C}^\infty(\overline{\Omega}) = \bigcap_{k=0}^\infty W^s(\Omega),\]
and further, the usual Fr\'{e}chet topology on $\mathcal{C}^\infty(\overline{\Omega})$  (given by the $\mathcal{C}^s$-norms)
coincides with the Fr\'{e}chet topology given by the $W^s$-norms. 

\medskip

We denote, as usual, by $\mathcal{C}^\infty_0(\Omega)$ the space of smooth functions on $\Omega$ with compact support,
and let  $W^s_0(\Omega)$ be the closure of $\mathcal{C}^\infty_0(\Omega)$ in $W^s(\Omega)$.

\subsection{Tensor Products.} We recall some algebraic notation which will facilitate working with functions on product
domains. For $j=1,\dots, k$, let $\Omega_j$ be a  domain in $\rl^{N_j}$, and let $\Omega$ be the product
$\Omega_1\times\Omega_2\times \ldots \times\Omega_k$, which is a domain in $\rl^N$, where $N=N_1 + N_2 + \ldots N_k$. If $f_j$ is a complex-valued
function on $\Omega_j$,  we denote by
\[ f_1\tensor\dots\tensor f_k\]
(the {\em tensor product} of the $f_j$'s)
the function on $\Omega$ defined by
\[ f(z_1,\dots,z_k)= f(z_1)\cdot f(z_2)\cdot \dots\cdot f(z_k)=\prod_{j=1}^k f(z_j),\]
where $z_j\in \Omega_j\subset\rl^{N_j}$. 

\medskip

For each $j$, let $X_j$ be a complex vector space of functions on $\Omega_j$. The {\em algebraic tensor product} 
\[ X_1\tensor \dots\tensor X_k = \bigotimes_{j=1}^k X_j\]
is a complex vector space of functions on $\Omega$.  By definition, every element of the algebraic tensor product 
may be written as a finite linear combination of tensor products of the form $f_1\tensor\dots\tensor f_k$, where $f_j\in X_j$.
We now recall the following basic fact:
 
 \begin{result}[{\cite[p.~369]{hor}}]\label{res-tensor}(1) The algebraic tensor product $\bigotimes_{j=1}^k\mathcal{C}^\infty_0(\Omega_j)$ 
is dense in the $\mathcal{C}^\infty$-topology
in $\mathcal{C}^\infty_0(\Omega)$ .

(2) Further, the algebraic tensor product $\bigotimes_{j=1}^k\mathcal{C}^\infty(\overline{\Omega}_j)$ is dense in 
$\mathcal{C}^\infty(\overline{\Omega})$ (in the Fr\'{e}chet topology.)
\end{result}

\noindent If $\mathsf{H}_j$ is a Hilbert space of functions on $\Omega_j$,  then the algebraic tensor product $\mathsf{H}=\bigotimes_{j=1}^k {\mathsf H}_j$ 
comes with a natural inner product,   given on functions which are tensor products as
\[ (f_1\tensor\dots\tensor f_k, g_1\tensor\dots\tensor g_k)_{\mathsf{H}}= \prod_{j=1}^k (f_j,g_j)_{\mathsf{H}_j}.\]
It is easy to verify that this extends by linearity to a  consistently defined inner product on $\mathsf{H}$. The completion of $\mathsf{H}$
with respect to this inner product is by definition, the {\em Hilbert Tensor Product}, which is a Hilbert space denoted by
\[  \mathsf{H}_1\csor \mathsf{H}_2\csor \dots \csor\mathsf{H}_k= {\widehat{\bigotimes}}_{j=1}^k \mathsf{H}_j.\]
 
\noindent If we take $\mathsf{H}_j$ to be the space $L^2(\Omega_j)$, it is not difficult to to show that 
\[ {\widehat{\bigotimes}}_{j=1}^k L^2(\Omega_j) = L^2(\Omega).\]

\noindent If for each $j$, we are also given a Hilbert space  $\mathsf{H}_j'$  of functions on a set $\Omega_j'$, and  bounded  linear maps $T_j:\mathsf{H}_j\to \mathsf{H}_j'$,   
we can define a map 
\[
T_1\tensor \dots \tensor T_k : \bigotimes_{j=1}^k\mathsf{H}_j \rightarrow \bigotimes_{j=1}^k\mathsf{H}_j'
\] 
between algebraic tensor products by setting 
\[
(T_1\tensor \dots  \tensor T_k)(h_1\tensor\dots \tensor h_k) = T_1(h_1)\tensor\dots\tensor T_k(h_k)
\]
on tensor products and extending linearly. This extends by continuity to a bounded linear map 
\[
T_1\csor\dots \csor T_k : {\widehat{\bigotimes}}_{j=1}^k \mathsf{H}_j \rightarrow {\widehat{\bigotimes}}_{j=1}^k \mathsf{H}_j'.
\]

\subsection{Partial Sobolev Spaces}
Let $s\geq 0$ be an integer, and define:
\begin{equation}\label{eq-pssdef} 
\pss^s(\Omega)= {\widehat{\bigotimes}}_{j=1}^kW^s(\Omega_j),
\end{equation}
and
\begin{equation}\label{eq-pss0def} 
\pss^s_0(\Omega)= {\widehat{\bigotimes}}_{j=1}^kW^s_0(\Omega_j),
\end{equation}
These spaces were studied in detail under the name ``Partial Sobolev Spaces" in 
\cite{prod}. We recall the following  simple property:
\begin{lem}If each $\Omega_j$ has  a Lipschitz boundary, we have
\begin{equation}\label{eq-pssinc}
 W^{ks}(\Omega)\subsetneq \pss^s(\Omega) \subsetneq W^s(\Omega),
\end{equation}
and
\begin{equation}\label{eq-pss0inc}
 W^{ks}_0(\Omega)\subsetneq \pss^s_0(\Omega) \subsetneq W^s_0(\Omega),
\end{equation}
where all inclusions are continuous.
\end{lem}

\noindent The proof of \eqref{eq-pssinc}, which is based on an explicit description of the norm on $\pss^s(\Omega)$,
 may be found in \cite{prod} (see pp.~992-993,  especially Lemma~5.1.)  The proof of  \eqref{eq-pss0inc} follows
from \eqref{eq-pssinc} using part (1) of Result~\ref{res-tensor} above.



\section{Bergman spaces on product domains}\label{sec-bergman}

\noindent We now specialize the considerations of the previous section to the case of domains in complex space. For $j=1,\dots, k$,  let 
$\Omega_j\subset\cx^{n_j}$ be a domain and  let $\Omega\subset \cx^n$ be the product domain $\Omega_1\times\Omega_2\times \ldots \times\Omega_k$,
where $n=n_1 + n_2 + \ldots + n_k$. For any domain $D\subset\cx^n$, let $\mathcal{H}(D) = L^2(D)\cap \mathcal{O}(D)$ be the {\em Bergman space}. Let  $P_j:L^2(\Omega_j)\to 
\mathcal{H}(\Omega_j)$ be the orthogonal projection, 
 known as the {\em Bergman projection}, and similarly
 $P:L^2(\Omega)\to \mathcal{H}(\Omega)$ is the Bergman projection on the product. 
 We begin by noting the following facts,  the crucial among them being the fact that Condition R is stable under formation of 
 products.

\begin{lem}\label{lem-miscprod} Assume that each $\Omega_j$ is bounded and pseudoconvex. Then:

(1) The Bergman space  on $\Omega$ may be represented as a Hilbert  tensor
product:
\[ \mathcal{H}(\Omega) = {\widehat{\bigotimes}}_{j=1}^k\mathcal{H}(\Omega_j).\]
 
 (2) The Bergman projection $P$ on $\Omega$ is represented  as
 \[ P = P_1 \csor \dots \csor P_k.\]

 (3) If each domain  $\Omega_j$ is  Lipschitz and  satisfies Condition R, then so does the product $\Omega$.

\end{lem}

\begin{proof}For (1) and (2)  we refer to \cite[Corollary~4.6]{prod} (when $\Omega_1$ and $\Omega_2$ have Lipschitz boundaries) 
or to \cite[Theorem~1.2]{spec} (for the general case.)

\medskip

For (3), note that the hypothesis of Condition R on each $\Omega_j$  implies that for each non-negative integer $s$,
there are non-negative integers $m_j(s)$, $j=1,\dots,k$,  such that  $P_j$ is continuous from $W^{m_j(s)}(\Omega_j)$ to $W^s(\Omega_j)$.
If $m(s)= \max(m_1(s),\dots, m_k(s))$,  then each $P_j$  is continuous from $W^{m(s)}(\Omega_j)$ to $W^s(\Omega_j)$. It follows from (2) 
now that the Bergman projection $P$ is continuous from  the Partial Sobolev space  $\pss^{m(s)}(\Omega)$ 
to the Partial Sobolev space  $\pss^s(\Omega)$,  as defined in \eqref{eq-pssdef}.
It now follows from \eqref{eq-pssinc}  that the Bergman projection $P$ maps $W^{km(s)}(\Omega)$ to $W^s(\Omega)$ continuously 
for each non-negative integer $s$, and this  shows that $\Omega$ also satisfies Condition R.
\end{proof}

On a domain $\Omega\subset\cx^n$, let $\mathcal{H}^s(\Omega) = W^s(\Omega)\cap \mathcal{O}(\Omega)$,
and recall that $\mathcal{H}^\infty(\Omega)$ has been defined earlier as the space $\mathcal{C}^\infty(\overline{\Omega})\cap\mathcal{O}(\Omega)$
of holomorphic functions smooth up to the boundary.  A crucial ingredient in the proof of the Fefferman--Bell--Ligocka theorem
is the following fact.

\begin{result}[{\cite[Lemma~2]{B2} and \cite[Lemma~6.3.9]{CS}}]\label{res-bell}
If  $\Omega\Subset \cx^n$ has smooth boundary,
then for each integer $s\geq 0$, there is an integer $\nu(s)\geq 0$, and an operator $\Phi^s$ bounded from $W^{s+\nu(s)}(\Omega)$ to $W^s_0(\Omega)$
such that $P\Phi^s =P$, where $P$ denotes the Bergman projection on $\Omega$. If $\Omega$ satisfies Condition R (or more generally, if
the space $\mathcal{H}^\infty(\Omega)$  is dense in the Bergman space $\mathcal{H}(\Omega)$)
 then  $\Phi^s$ maps $\mathcal{H}^s(\Omega)$ to $W^s_0(\Omega)$.
\end{result}

\noindent Indeed, $\Phi^s$ is realized as a differential operator of degree $\nu(s)= s(s+1)/2$
with coefficients smooth up to the boundary on $\Omega$. For conciseness, for an integer $s\geq 0$,  let us refer to an operator with the properties
stated in the conclusion of the above result as a {\em Bell operator of order $s$.} We now note that this property is also inherited by products from factors:
\begin{lem} \label{lem-bellprod} Let $\Omega=\Omega_1\times\dots\times\Omega_k\subset\cx^n$  be a product of smoothly bounded domains 
$\Omega_j$. Then $\Omega$ admits a Bell operator. Further, this operator maps the space
 $\widetilde{\mathcal{H}}^s(\Omega)=\pss^s(\Omega)\cap\mathcal{O}(\Omega)$ to $\pss^s_0(\Omega)$.
\end{lem}
\begin{proof} Let $s\geq 0$,  and let $\Phi^s_j:W^{\nu(s)+s}(\Omega_j)\to W^s_0(\Omega_j)$,  be the Bell operator on $\Omega_j$,
where $\nu(s)=s(s+1)/2$. 
Define an operator on functions on $\Omega$ by:
\[ \Phi^s= \Phi^s_1\csor \Phi^s_2\csor \dots\csor \Phi^s_k.\]
Then $\Phi^s$ is continuous from $\pss^{\nu(s)+s}(\Omega)$ to $\pss^s_0(\Omega)$, and therefore by \eqref{eq-pssinc} and \eqref{eq-pss0inc}
from $W^{N(s)+s}(\Omega)$ to $W^s_0(\Omega)$ where $N(s)= k\nu(s)+ (k-1)s$.  Also, denoting by $P_j$ the Bergman projection on $\Omega_j$ and 
by $P$ the Bergman projection on $\Omega$, we have
\begin{align*}
P\Phi^s&= (P_1\csor\dots\csor P_k)(\Phi^s_1\csor\dots\csor\Phi^s_k)\\
&=P_1\Phi^s_1\csor\dots\csor P_k\Phi^s_k\\
&=P_1\csor\dots\csor P_k\\
&=P.
\end{align*}
The last statement is obvious.
\end{proof}

\subsection{Estimates in Bergman spaces}
We now note that some well-known estimates on Bergman functions continue to hold on product domains. Denote by $\norm{\cdot}_s$ the norm of a function in the space $W^s$. 
We will need the Sobolev space of negative order $W^{-s}$.  Recall that this is the dual of $W^s_0$, and for a domain $D$, the negative Sobolev norm of a function $g$ on $D$ is
defined by

\[ \norm{g}_{-s}=\sup_{\substack{\phi\in\mathcal{C}^\infty_0(D)\\\norm{\phi}_s=1}}\abs{\int_D g\overline{\phi}}.\]

\noindent Now suppose that $D$ is a product of smoothly bounded domains $D_1 \times D_2 \times \ldots \times D_k$. We define $\pss^{-s}(D)$ 
(a partial Sobolev space of negative index) to be the dual of $\pss^s_0(D)$. The norm in $\pss^{-s}(D)$  is given by
\[ \norm{g}_{\widetilde{W}^{-s}(D)}
=\sup_{\substack{\phi\in\mathcal{C}^\infty_0(D)\\\norm{\phi}_{\pss^s(D)}=1}}\abs{\int_D g\overline{\phi}}. \]
The following properties are easy to see:

\begin{lem}(1) $\widetilde{W}^{-s}(D)= W^{-s}(D_1)\csor W^{-s}(D_2)\csor\dots\csor W^{-s}(D_k)$,\\
(2) $W^{-s}(D)\subsetneq \pss^{-s}(D) \subsetneq W^{-ks}(D)$
\end{lem}

\begin{proof}
These follow on taking duals in the relations \eqref{eq-pss0def} and \eqref{eq-pssinc} respectively.
\end{proof}

\noindent The following is an analog of well-known facts for smoothly bounded domains (see \cite{B2}). Related duality issues on domains with nondegenerate corners have been studied by 
Barrett in \cite{Ba}.

\begin{lem}\label{lem-prodest}
Let $D\subset\cx^n$ be a domain represented as $D_1\times D_2 \times \ldots \times D_k$, where 
each $D_k$ is smoothly bounded and satisfies Condition R. Then 
\begin{enumerate}
\item   if $v,g\in \mathcal{H}(D)$,  we have
for each integer $s\geq 0$:
\[ \abs{\int_D v \overline{g}}\leq C \norm{v}_{ks}\norm{g}_{-s}.\]
\item If $g\in \mathcal{H}(D)$, then for  $s>n$ we have
\begin{equation}\label{eq-est1}
C_1\norm{g}_{-s-n-1}\leq \sup_{z\in D}\left( \abs{g(z)}d(z)^s\right)\leq C_2\norm{g}_{-s+n},
\end{equation}
where $C_1,C_2$ are constants independent of $g$, and $d(z)=\dist(z,\partial D)$. 
\end{enumerate}
\end{lem}

\begin{proof} By Lemma~\ref{lem-miscprod}, the domain $D$ satisfies Condition R, and  therefore, functions in
 $\mathcal{H}^\infty(D)$ are dense in the Bergman space $\mathcal{H}(D)$. Therefore, 
 it suffices to prove the result when $v$ and $g$ are in $\mathcal{H}^\infty(D)$.
By Lemma~\ref{lem-bellprod}, there is a Bell operator $\Phi^s$  which maps $W^{s+N(s)}(D)$ to
$W^s_0(D)$, satisfies $P\Phi^s=P$, and which maps $\widetilde{\mathcal{H}}^s(D)$ to $\pss^s_0(D)$.  We have
\begin{align*}\abs{\langle v, g\rangle }&\leq \abs{\langle \Phi^s v, g \rangle}\\
&\leq \norm{\Phi^s v}_{\widetilde{W}^s(D)} \norm{g}_{\widetilde{W}^{-s}(D)}\\
&\leq C \norm{v}_{\widetilde{W}^s(D)} \norm{g}_{\widetilde{W}^{-s}(D)}.
\end{align*}
Since the inclusions $W^{ks}(D)\subset \pss^s(D)$ and $W^{-s}(D)\subset \pss^{-s}(D)$ are continuous,
the result  (1) follows.

\medskip

For the inequalities in \eqref{eq-est1}, we note that the proof in the standard smoothly bounded situation 
carries over word-by-word for product domains. We recall that the Sobolev embedding theorem continues to hold 
in the Lipschitz domain $D$.
\end{proof}

\section{Estimates on the distance to the boundary for proper mappings}
\label{sec-estimates}
\begin{prop}\label{prop-dist}Let $D$ and $G$ be domains which may each be represented as a product of smoothly bounded
pseudoconvex domains. If $f:D\to G$ is a proper holomorphic map, then there is a $C>0$ and an $0<\eta<1$ such that for $z\in D$, we have
\[ \frac{1}{C}\dist(z,\partial D)^{\frac{1}{\eta}}\leq\dist(f(z), \partial G)\leq C \dist(z,\partial D)^\eta.\]
\end{prop}

\noindent The proof is based on the application of a version of the Hopf lemma to bounded \psh exhaustion functions.
Let $\Omega$ be a domain in $\cx^n$. Recall that a continuous \psh function $\psi<0$ on $\Omega$ is referred to as a 
{\em bounded \psh exhaustion function} if for every $\epsilon>0$, the set $\{\psi < -\epsilon\}$ is relatively compact in $\Omega$. It follows
that such a function $\psi$  extends to a continuous  function on $\overline{\Omega}$ which vanishes on the boundary $\partial\Omega$.
The classical Hopf lemma needs at least $\mathcal{C}^2$ smoothness of the boundary, and applies to subharmonic functions (see \cite{hopf}.)
We will need the following version of the Hopf lemma, which applies to   {\em \psh} functions on product domains:
\begin{prop}\label{prop-hopf}
If  $\Omega$ is a domain which is the product of smoothly bounded domains, then for every continuous bounded  plurisubharmonic exhaustion
$\psi$ of $\Omega$, there is a constant $C>0$ such that we have
\[\abs{\psi(z)}\geq C \,\dist(z,\partial\Omega).\]
\end{prop}

\noindent We postpone the proof of  Proposition~\ref{prop-hopf}, but now prove  Proposition~\ref{prop-dist} using it:

\begin{proof}[Proof of Proposition~\ref{prop-dist}] 
According to a theorem of Diederich and Fornaess \cite{difo}, if  the pseudoconvex domain 
$\Omega$ has $\mathcal{C}^2$ boundary and is bounded,
there is a smooth defining function $\rho$
of $\Omega$, and $\eta$ with $0<\eta< 1$ such that $\psi= - (-\rho)^\eta$ is  a bounded \psh exhaustion. If 
$\Omega= \Omega_1\times\Omega_2\times \dots \Omega_k$ is a product of smoothly bounded pseudoconvex domains,  and $\psi_j= -(-\rho_j)^{\eta_j}$
is the Diederich-Fornaess bounded \psh exhaustion on $\Omega_j$, we can define a continuous  bounded \psh exhaustion $\lambda$
on the product $\Omega$, by setting for a point $z=(z_{(1)}, \dots,z_{(k)})$,  where $z_{(j)}\in \Omega_j$,

\begin{equation}\label{eq-newlambda}\lambda(z) = \max_{1\leq  j\leq k} \psi_j(z_{(j)}).\end{equation} 

\noindent Note that by definition, for each $z\in \Omega$,
\begin{align}
\abs{\lambda(z)}&= -\lambda(z)\nonumber\\
&= \min_{1\leq j \leq k} (-\rho_j(z))^{\eta_j}\nonumber\\
&\leq C \dist(z,\partial\Omega)^\eta,\label{eq-bddpsh}
\end{align}
where $\eta$ is the minimum of the $\eta_j, j=1,\dots, k$. Applying  \eqref{eq-bddpsh} to the product domain $D$, we  obtain a \psh exhaustion function $\lambda$ such that
\begin{equation}\label{eq-lambda1} \abs{\lambda(z)}<  k \,\dist(z,\partial D)^\eta\end{equation}
for some constant $k>1$. Let $F_1, F_2, \ldots, F_m$ be the branches of $f^{-1}$ which are locally well defined holomorphic functions on $G \setminus Z$, where
$Z = \{f(z) : \det f'(z) = 0\} \subset G$ is a codimension one subvariety. Then
\[
\psi = \max \{ \lambda \circ F_j : 1 \le j \le m \}
\]
is a bounded continuous plurisubharmonic function on $G \setminus Z$ which admits an extension as a plurisubharmonic exhaustion of $G$; we will retain $\psi$ as the notation 
for this extension. Therefore, for each $1 \le j \le m$ and $w \in G$, we have
\[
- \lambda \circ F_j(w)  \ge -\psi(w)  = \vert \psi(w) \vert \ge C\;{\rm dist}(w, \partial G)
\]
where the last inequality follows from Proposition~\ref{prop-hopf}. This can be rewritten as
\[
\vert \lambda(z) \vert = -\lambda(z) \ge C\;{\rm dist}(f(z), \partial G)
\]
for all $z \in D$. Combining this with \eqref{eq-bddpsh}, the right-half of the result follows.

\medskip

For the left half, we  construct as before a bounded \psh exhaustion $\mu$ of the product domain $G$
satisfying 
\[ \abs{\mu(w) } \leq \ell \; \dist(w,\partial G)^\theta\]
with appropriate positive constants $\ell, \theta$. Then $\phi = \mu\circ f$ is a \psh exhaustion of the domain $D$, and we can apply the Hopf lemma
to $\phi$ to conclude that 
\[ \abs{\mu(f(z))}\geq C \; \dist(z,\partial D)\]
Combining the two estimates, the left-half of the result follows.
\end{proof}

\subsection{Rolling analytic discs }\label{sec-rolling} We now want to prove Proposition~\ref{prop-hopf},
and we will do so in a slightly more general context than required by our application.
The requirement of $\mathcal{C}^2$-smoothness of the boundary in the classical Hopf Lemma arises since
we need to roll a ball of fixed radius in the domain on the boundary with the ball in the domain, in such a way that every point 
of the boundary is touched by the ball. In the case of plurisubharmonic functions, we can replace the ball by an affine analytic disc,
i.e., the image in $\cx^n$ of the closed unit disc $\overline{\D}=\{\abs{\lambda}\leq 1\}\subset \cx$ under a map of the form
$ \lambda\mapsto z + \lambda v$ where $z\in \cx^n$ is the center and $v\in \cx^n$ is a vector whose length $\abs{v}$ is the radius
of the disc. The {\em boundary} of the disc is the image of the boundary $\partial\D =\{\abs{\lambda}=1\}$ of $\D$
under such a mapping. 

\medskip

In order to roll a disc on the boundary of a domain $\Omega\subset \cx^n$, we will need some geometric 
conditions on the domain. Let $U$ be a neighborhood of $\partial\Omega$ in $\cx^n$.
The domain $U \cap \Omega$ whose boundary consists of two connected components, namely $\partial \Omega$ and $B=\partial U\cap \Omega$ 
will be relevant to us. Suppose that there is a constant $\theta = \theta(\Omega) > 0$ and points
$\kappa(z) \in B$ and $\zeta(z) \in \partial\Omega$  for every $z \in U \cap \Omega$ such that the following hold:

\begin{enumerate}
\item[(D1)] The points $\zeta(z), z$ and $\kappa(z)$ are collinear and $z$ is situated between $\zeta(z)$ and $\kappa(z)$.
\item[(D2)] $\zeta(z)$ is a nearest point to $z$ on $\partial\Omega$, i.e., $\abs{ \zeta(z) - z} = {\rm dist}(z, \partial \Omega)$. (It is {\em not} assumed that there is a unique nearest point to $z$ on $\partial\Omega$.)
\item[(D3)] Let $\alpha_z$ denote the affine disc centered at $\kappa(z)$ and whose  boundary passes through the point $\zeta(z)$, so that
\begin{equation}\label{eq-alphaz}\alpha_z(\lambda)= \kappa(z)+ \lambda (\zeta(z)-\kappa(z)).\end{equation}
Then $\alpha_z(\D)$  is contained in $\Omega$.
\item[(D4)] There exists a neighborhood of $\partial \Omega$ in $\mathbb C^n$, say $V$ which is compactly contained in $U$ such that the part of the boundary of $\alpha_z$ 
that does not lie in $V$ subtends an angle of at least $\theta > 0$ at the center $\kappa(z)$. In other words, the set 
\[ \partial\D \cap \alpha_z^{-1}(\Omega\setminus V)\]
has angular length at least $\theta$.
\end{enumerate}

\noindent {\it Remarks:} First, note that there exists an $\eta > 0$ such that ${\rm dist}(\Omega \setminus U, \partial \Omega) \ge \eta$. This follows since $U$ is a neighborhood of 
$\partial \Omega$ and the assumption that $\Omega$ is bounded. The same reasoning shows that $\Omega \setminus V$ and $\partial \Omega$ also 
have a positive distance between them; in fact $\Omega \setminus U$ is compactly contained in $\Omega \setminus V$. In particular, we see 
that $\abs{\kappa(z) - \zeta(z)} \ge \eta$. Second, if there is a neighborhood $V$ of $\partial \Omega$ that satisfies (D4), then all smaller neighborhoods 
$W \subset V$ of $\partial \Omega$ will also satisfy (D4).

\medskip

\noindent Let us say that inside a domain satisfying the conditions (D1)--(D4) above we can {\em roll an analytic disc}. The following is a  
more general version of the Hopf lemma which we will prove in the next section:

\begin{thm}[Hopf lemma for \psh functions] \label{thm-hopflemma} Let $\Omega\subset\cx^n$ be a domain inside which
it is possible to roll an analytic disc, and let $\psi $ be a continuous bounded \psh exhaustion of $\Omega$. Then there is a constant
$C>0$ such that 
\begin{equation}\label{eq-hopf} \abs{\psi(z)}\geq C \,\dist(z,\partial\Omega).\end{equation}
\end{thm}

\noindent In order to deduce Proposition~\ref{prop-hopf} from the abstract statement of Theorem~\ref{thm-hopflemma}, we need the following two
facts:
\begin{lem}\label{lem-smoothrolling}It is possible to roll an analytic disc inside a domain with $\mathcal{C}^2$-smooth boundary.
\end{lem}
\begin{proof}
Let $\Omega$  be a domain with $\mathcal{C}^2$-smooth boundary in $\cx^n$.  Let $\eta=\eta(\Omega)>0$ be so small that 
the domain 
\[
U = \{z\in \mathbb C^n :  \dist(z,\partial\Omega)<2\eta\}
\]
is a tubular neighborhood of $\partial\Omega$ in $\Omega$. We let
\[ 
V=  \{z\in \mathbb C^n : \dist(z,\partial\Omega)<\eta\}.
\]
For any $z\in U \cap \Omega$, let $\zeta(z)$ be the unique point on $\partial\Omega$ closest to $z$, and 
let $\kappa(z)$ be the point on the straight line in $\overline{\Omega}$ through $\zeta(z)$ and $z$ 
which is at a distance $2\eta$ from $\zeta(z)$. Then conditions (D1) and (D2) are clearly satisfied. Also the ball
$B(\kappa(z),\eta)$ is contained in $\Omega$ and touches $\partial\Omega$ only at $\zeta(z)$. Since $\alpha_z(\D)\subset B(\kappa(z),\eta)$,
condition (D3) follows. Finally, note that the angular length of the part of circle mapped by $\alpha_z$ outside $V$, i.e., $\partial \D \cap \alpha_z(\Omega \setminus V)$ is a 
continuous and positive function of $z$. The existence of $\theta(\Omega)$ now follows by compactness.
\end{proof}

\noindent It is easy to extend the rolling of analytic discs to products:

\begin{lem}\label{lem-productrolling}  If it is possible to roll a disc inside each of two domains,  it is possible to roll a disc inside their product.
\end{lem}
\begin{proof}Let $\Omega_1$, $\Omega_2$ be domains inside each of which it is possible to roll an analytic disc, and 
let $U_j, V_j, \kappa_j, \zeta_j$ denote the objects posited for each $j=1,2$ at the beginning of section~\ref{sec-rolling}. We define the corresponding 
objects on the product $\Omega=\Omega_1\times \Omega_2$ in the following way. We let the open neighborhood $U$ of $\partial\Omega$ 
be  the union $(\Omega_1\times U_2) \cup (U_1\times \Omega_2)$ and likewise let $V = (\Omega_1 \times V_2) \cup (V_1 \times \Omega_2)$. We also 
let $\theta(\Omega)=\min(\theta(\Omega_1),\theta(\Omega_2)) > 0$.

\medskip

Now let $z\in U \cap \Omega$.  Writing $z=(z_1,z_2)\in\Omega_1\times\Omega_2$, we have $\dist(z,\partial\Omega)= \min\left(\dist(z_1,\partial\Omega_1),
\dist(z_2,\partial\Omega_2)\right)$. After swapping the indices if  required, we can assume that $\dist(z,\partial\Omega)= \dist(z_2,\partial\Omega_2)$.
We now let $\zeta(z)= (z_1,\zeta_2(z_2))$ and $\kappa(z)=(z_1,\kappa_2(z_2))$. Let
\[
\alpha_{z_2}(\lambda) = \kappa_2(z_2) + \lambda(\zeta_2(z_2) - \kappa_2(z_2))
\]
be the affine analytic disc in $\Omega_2$ that satisfies the conditions (D1) to (D4) in $\Omega_2$. Define 
\[
\alpha_z(\lambda) = (z_1 , \alpha_{z_2}(\lambda))
\]
which is an affine analytic disc in $\Omega_1 \times \Omega_2$ whose center is at $\alpha_z(0) = (z_1, \alpha_{z_2}(0)) = (z_1, z_2)$ and 
which contains $\zeta(z)$ in its boundary. Properties (D1) through (D4) for $\alpha_z$ now follow from the fact 
that they hold for $\alpha_{z_2}$ in $\Omega_2$.
\end{proof}

\noindent Combining this result and Lemmas~\ref{lem-smoothrolling} and \ref{lem-productrolling} with Theorem~\ref{thm-hopflemma},
we obtain Proposition~\ref{prop-hopf}.  It only remains to give a proof of  Theorem~\ref{thm-hopflemma}.

\subsection{Estimate on bounded \psh exhaustion.}

\begin{proof}[Proof of Theorem~\ref{thm-hopflemma}]
Let $U, V$ be neighborhoods of $\partial \Omega$ in $\mathbb C^n$ with $V$ compactly contained in $U$ and $\theta = \theta(\Omega) > 0$ be such that 
properties (D1)--(D4) hold on $\Omega$. It is sufficient 
to show that \eqref{eq-hopf} holds for $z\in U \cap \Omega$.  Fixing 
such a $z$, we let $\alpha_z$ be the analytic disc in \eqref{eq-alphaz}, and let $u_z$ be the continuous function defined on $\overline{\D}$ given by
$u_z= \psi\circ\alpha_z$, so that $u_z $ is  subharmonic on $\D$, is less than or equal to zero on $\overline{\D}$, and equal to zero  at $1\in \overline{\D}$.
 Let $h_z$ be the harmonic majorant of $u_z$ on the disc $\D$, i.e., the solution of the Dirichlet problem on the disc with boundary data $u_z|_{\partial\D}$.
 Let $\lambda\in \D$. We have
 \begin{align*}
 u_z(\lambda)&\leq h_z(\lambda)\\
 &= \frac{1}{2\pi}\int_{0}^{2\pi} \frac{1-\abs{\lambda}^2}{1- 2 \Re(\lambda e^{-i\phi})+\abs{\lambda}^2}u(e^{i\phi})d\phi.\\
& =\frac{1-\abs{\lambda}^2}{(1+\abs{\lambda})^2}\cdot\frac{1}{2\pi}\int_0^{2\pi} u(e^{i\phi})d\phi\\
 &=h_z(0)\cdot\frac{1-\abs{\lambda}}{1+\abs{\lambda}}\\
 &\leq\frac{h_z(0)}{2}\left(1-\abs{\lambda}\right).
 \end{align*}
 Let $\lambda_z = \displaystyle{\frac{\abs{z-\kappa(z)}}{\abs{\zeta(z)-\kappa(z)}}}$.
  Since the three points $\kappa(z) , z$ and $\zeta(z)$ are collinear in 
 this order, it follows that  $z= \kappa(z)+ \lambda_z(\zeta(z)-\kappa(z))$.   Therefore,
 \begin{align*}
 \psi(z)&=u_z(\lambda_z)\\
 &\leq \frac{h_z(0)}{2}\left( 1 - \frac{\abs{z-\kappa(z)}}{\abs{\zeta(z)-\kappa(z)}}\right)\\
 &\leq \frac{h_z(0)}{2}\frac{\abs{\zeta(z)-\kappa(z)}-\abs{z-\kappa(z)}}{\abs{\zeta(z)-\kappa(z)}}\\
 &\leq \frac{h_z(0)}{2}\frac{\dist(z,\partial\Omega)}{\eta} \end{align*}
where we have used the observation made earlier that $\abs{\kappa(z) - \zeta(z)} \ge \eta$ for some $\eta = \eta(\Omega) > 0$. Noting that 
$\psi$, $u_z$ and $h_z$ are negative, we have
\begin{align*} 
\abs{h_z(0)}&= \frac{1}{2\pi}\int_0^{2\pi}\abs{u(e^{i\phi})}d\phi\\
&= \frac{1}{2\pi}\int_0^{2\pi}\abs{\psi(\alpha_z(e^{i\phi}))}d\phi.\\
\end{align*}
Denote by $\mu$ the minimum value of the continuous function $\abs{\psi}$ on the compact set $\Omega\setminus V$. Since it is 
possible to roll a disc inside $\Omega$,  there is a subset of measure at least $\theta$ of $[0,2\pi]$ which 
is mapped by $\alpha_z$ into $\Omega\setminus V$. Therefore, it follows that $\abs{h_z(0)}\geq \mu\theta/2\pi$, independently of $z$.
Therefore, we have $\abs{\psi(z)}\geq C \dist(z,\partial\Omega)$ with $C=\mu\theta/4\pi\eta$ independent of $z$.
\end{proof}


\subsection{Some more general results}
We note here a couple of statements more general than Proposition~\ref{prop-dist} which can be proved starting from
Proposition~\ref{prop-hopf}:

\begin{prop}\label{prop-gendist}(1) Let $D$ be a Lipschitz pseudoconvex domain, and let $G$ be a domain which can be represented 
as a product of smoothly bounded domains.  If  $f:D\to G$ is a proper holomorphic map,
then there is a $C>0$, and an $0<\eta<1$  such that for each $z\in D$, we have
\[ \dist(f(z), \partial G)\leq C \;\dist(z,\partial D)^\eta.\]
(2) Let $D$ be a domain which can be represented as product of smoothly bounded pseudoconvex domains, and 
let $G$ be a Lipschitz domain. If $f:D\to G$ is a proper holomorphic map, 
then there is a $C>0$, and an $0<\eta<1$  such that for each $z\in D$ we have
\[  \frac{1}{C}\dist(z,\partial D)^{\frac{1}{\eta}}\leq\dist(f(z), \partial G)\]
\end{prop}

\noindent In order to prove Proposition~\ref{prop-gendist}, we need a \psh exhaustion function on a Lipschitz domain. 
It has recently been shown by Harrington (see \cite{harrington}) that it is possible to construct
even a {\em strictly} \psh bounded exhaustion functions on arbitrary  bounded {\em Lipschitz}
pseudoconvex domains in $\cx^n$.  More precisely, we have the following:
\begin{result}\label{res-harring}
Let $\Omega\Subset\cx^n$ be a bounded pseudoconvex domain with Lipschitz boundary. Then for some $0<\eta<1$, there exists
a negative strictly plurisubharmonic function $\lambda$ on  $\Omega$  and a constant $k>1$  such that 
\[ \frac{1}{k}\dist(\cdot,\partial\Omega)^\eta < -\lambda < k \,\dist(\cdot,\partial\Omega)^\eta.\]
\end{result} 
\noindent The proof of Proposition~\ref{prop-gendist} follows in the same way as that of  Proposition~\ref{prop-dist},
using  Result~\ref{res-harring} to construct the required  bounded \psh exhaustions.


\section{Proof of Theorem~\ref{thm-main}}
\label{sec-proof}
We break the proof up into several steps. Let $f:D\to G$ be a proper holomorphic map
as in the statement of Theorem~\ref{thm-main}, and  let $u=\det(f')$ be the {\em complex}
Jacobian determinant of the mapping $f$. 

\begin{lem}\label{lem-gof} For every integer $s\geq 0$, there is an integer $j(s)\geq 0$ such 
that the mapping from $L^2(G)$ to $L^2(D)$ given by 
\[ g\mapsto u\cdot (g\circ f)\]
is continuous from $W^{s+j(s)}_0(G) $ to $W^s_0(D)$.
\end{lem}

\begin{proof} 
Following \cite{Be1}, \cite{CS} and \cite{krantz}, it suffices to show that there is a uniform constant $C > 0$ such that
\[
\norm{u \cdot (g \circ f)}_{W^s_0(D)} \le C \norm{g}_{W^{s + j(s)}_0(G)}
\]
for all $g \in \mathcal C^{\infty}_0(G)$ and for an appropriately chosen $j(s)$. Let $f = (f_1, f_2, \ldots, f_n)$ where $f_j \in \mathcal O (D)$ for $1 \le j \le n$ 
and let $d_1(\cdot) = \dist(\cdot, \partial D)$ and $d_2(\cdot) = \dist(\cdot, \partial G)$. Since $D, G$ are bounded, the Cauchy estimates imply that there is a uniform $C > 0$ such 
that
\[
\left \vert D^{\alpha} f_l (z) \right \vert \le C (d_1(z))^{-\abs{\alpha}}
\]
and
\[
\left\vert D^{\alpha} u(z) \right\vert \le C (d_1(z))^{-(\abs{\alpha} + 1)}
\]
for all multiindices $\alpha$ and $1 \le l \le n$. Now for every such $\alpha$ with $\abs{\alpha} \le s$ observe that
\begin{equation}
D^{\alpha} \left( u \cdot (g \circ f) \right) = \sum D^{\beta}u \cdot \left( (D^{\gamma} g) \circ f \right) \cdot D^{\delta_1} f_{i_1} \cdot D^{\delta_2} f_{i_2} \ldots \cdot 
D^{\delta_p} f_{i_p}
\end{equation}
where the sum extends over all $1 \le i_1, \ldots, i_p \le n$ and multiindices $\beta, \gamma, \delta_1, \ldots \delta_p$ with $\abs{\beta} \le \abs{\alpha}, \abs{\gamma} \le 
\abs{\alpha} - \abs{\beta}$ and $\abs{\delta_1} + \ldots + \abs{\delta_p} = \abs{\alpha} - \abs{\beta}$. It follows that for $z \in D$
\[
\left \vert D^{\beta}u \cdot D^{\delta_1} f_{i_1} \cdot D^{\delta_2} f_{i_2} \ldots \cdot D^{\delta_p} f_{i_p}  \right \vert \le C d_1(z)^{-(\abs{\alpha}+1)}.
\]
For any smooth $\phi$ which is compactly supported in $G$, the Sobolev embedding theorem (which continues to 
hold on a Lipschitz domain thanks to the extension property stated in Section~\ref{sec-functions} above) combined with Taylor's theorem shows that for every $k \ge 1$
\[
\left \vert D^{\gamma} \phi(w) \right \vert \le C \norm{\phi}_{k + \abs{\gamma} + n + 1} \cdot d_2(w)^k
\]
for some $C > 0$ that is independent of $\phi$. Applying this to the compactly supported smooth function $g$ and using it in (5.1) we see that
\begin{align*}
\left \vert D^{\alpha} \left( u \cdot( g \circ f)  \right)(z)  \right \vert &\le C d_1(z)^{-(\abs{\alpha} + 1)} \cdot \norm{g}_{k + \abs{\alpha} + n + 1} \cdot d_2(f(z))^k\\
                                                                            &\le C \norm{g}_{k + s + n + 1} \cdot d_1(z)^{-\abs{\alpha}-1+k\eta}
\end{align*}
where the second inequality follows from the right half of the estimate in Proposition 4.1. By choosing $k > (s + 1)/\eta$ it follows that the mapping $g \mapsto u \cdot( g \circ f)$ is 
bounded from $W^{s + j(s)}_0(G)$ to $W^{s}_0(D)$ with $j(s) = k + n + 1$.
\end{proof}

\begin{lem}\label{lem1}
For each $h\in \mathcal H^{\infty}(G)$, we have 
\[  u\cdot (h\circ f)\in \mathcal{C}^\infty(\overline{D}).\]
In particular, $u \in \mathcal C^{\infty}(\overline D)$.
\end{lem}

\begin{proof}We adapt the proof in \cite{B2}, \cite{Be1} to the present situation.
Let $P:L^2(D)\to \mathcal H(D)$ and $Q: L^2(G)\to  \mathcal H(G)$ denote the 
Bergman projections  on the domains $D$ and $G$ respectively. 
It is known (see \cite{Be1}) that the following identity holds for any function $g\in L^2(G)$:
\begin{equation}\label{eq-p1p2}
P(u\cdot(g\circ f))= u\cdot (Q(g)\circ f).
\end{equation}

\medskip

Fix a positive integer $s\geq 0$, and let $h\in \mathcal H^{\infty}(G)$ be a holomorphic 
function on $G$ smooth up to the boundary. Since $D$ satisfies Condition R, there is an $m(s)$ such that if $\phi\in W^{s+m(s)}(D)$,
then $P\phi\in W^s(D)$. This follows from the fact, noted in Section~\ref{sec-functions} above that the Fr\'{e}chet topology on
$\mathcal{C}^\infty(\overline{G}_1)$ is also given by the Sobolev norms $W^k(D)$, with $k\in \mathbb{N}$. Further by Lemma~\ref{lem-gof} above, there 
is an integer $j'(s)= j(s+m(s))$ such that if $\phi\in W^{s+j'(s)}_0(D)$ then $u\cdot (\phi\circ f) \in W^{s+m(s)}(D)$.

\medskip

Now let $g=\Phi^{s+j'(s)}h$, where $\Phi^k$ denotes the Bell operator of order $k$ on $G$, whose existence follows from Lemma~\ref{lem-bellprod}
and Result~\ref{res-bell} above. Since $h$ is smooth up to the boundary, it follows that $g\in W^{s+j'(s)}_0(G)$, which in turn implies, using the 
estimates of the last paragraph,  that $P(u\cdot (g\circ f))\in W^s(D)$.

\medskip

Therefore plugging in this $g$ into \eqref{eq-p1p2} above, we see that the left hand side is in $W^s(D)$ whereas the right hand side is equal to 
$ u \cdot (Q(g) \circ f) = u \cdot ((Q \Phi^{s + j'(s)} h) \circ f) = u\cdot (h\circ f)$ since $h$ is holomorphic. It follows then that 
$u\cdot (h\circ f)\in W^s(D)$ for each 
integer $s\geq 0$, and consequently $u\cdot (h\circ f)\in\mathcal{C}^\infty(\overline{D})$. By taking $h \equiv 1$, it follows that $u \in \mathcal C^{\infty}(\overline D)$.
\end{proof}

\subsection{Symmetric functions of the branches}
The next step is to prove that $u$ vanishes to finite order at each point on $\partial D$. Following \cite{Be1}, we first show that any elementary symmetric function of the various 
branches of $f^{-1}$ is well defined near $\partial G$. More precisely, in the situation of Theorem 1.1, we have the following:

\begin{prop}\label{prop-symfun}
Let $h \in \mathcal H^{\infty}(D)$. Let $F_1, F_2, \ldots, F_m$ be the branches of $f^{-1}$ which are locally well defined holomorphic functions on $G 
\setminus Z$. Then the elementary symmetric functions of $h \circ F_1, h \circ F_2, \ldots, h \circ F_m$ extend to holomorphic functions on $G$ which are in $\mathcal 
C^{\infty}(\overline G)$. 
\end{prop}

\noindent The proof of Proposition~\ref{prop-symfun} will require the following result (cf. \cite[Fact~2]{Be1})

\begin{lem}\label{lem-EN} For each integer $s\geq 0$, there is a positive integer $N=N(s)$ and a positive constant $C=C(s)$ such that
\[ \norm{v\cdot (g\circ f)}_{-N} \leq C \;\sup_{z\in D}\abs{v(z)}\norm{g}_{-s}\]
for any $v\in \mathcal{H}(D) $ and any $g\in \mathcal{H}(G)$.
\end{lem} 

\begin{proof} We use the left half of Proposition~\ref{prop-dist}, i.e.,   $C^{-1} \dist(z,\partial D)^{\frac{1}{\eta}}\leq\dist(f(z), \partial G)$,
for some $C>0$ and $0<\eta<1$. Let $d_1(\cdot)= \dist(\cdot, \partial D)$ and $d_2(\cdot)= \dist(\cdot,\partial G)$, and 
choose $N$ so large that $n+N>(s+n)/\eta$. By successively using the two halves of \eqref{eq-est1}, we get:

\begin{align*}
\norm{v\cdot (g\circ f)}_{-N} &\leq C \sup_{z\in D}\abs{v(z)g(f(z))}d_1(z)^{n+N+1}\\
&\leq C \sup_{z\in D} \abs{v(z)} \norm{g}_{-s} \sup_{z\in D} \left( d_2(f(z))^{-s-n}d_1(z)^{n+N+1}\right)\\
&\leq C \norm{g}_{-s}\sup_{z\in D}\abs{v(z)}.
\end{align*}
\end{proof}

\begin{proof}[Proof of Proposition~\ref{prop-symfun}] For a positive integer $r$, set
\[ H_r = \sum_{j=1}^m (h\circ F_j)^r.\]
Since the elementary symmetric functions of $\{h\circ F_j\}_{j=1}^m$ may be written as polynomials in the 
functions $H_r$, it suffices to show that $H_r\in \mathcal{H}^\infty(\overline{G})$. Note that
$H_r$ is bounded and holomorphic in the complement of  $\{f(z) : u(z)=0\}$ (which is an analytic variety by  Remmert's proper mapping theorem),
so by the Riemann removable 
singularity theorem $H_r$ extends as a holomorphic function on $G$. In order to prove that $H_r\in \mathcal{C}^\infty(\overline{G})$, we will 
prove that for each multi-index $\alpha$, the  function $\partial^\alpha H_r / \partial z^\alpha$
is bounded on $G$.

\medskip

Recall that $G = G^1 \times G^2 \times \ldots \times G^l$, where each $G^j$ is smoothly bounded and pseudoconvex. 
Thanks to a classical result of Kohn (\cite{Ko}), for each integer $s\geq 0$, there is a $t_j>0$ such that for $t\geq t_j$, the projection $P_t^j$ from
$L^2(G^j)$ to $\mathcal{H}(G^j)$ in the inner product

\[ \langle v, g \rangle_{G_j, t} = \int_{G_j}v(z) \overline{g(z)}\exp(-t\abs{z}^2)\]

maps $W^s(G^j)$ to $\mathcal{H}^s(G^j)$. Let $t_*(s) = \max_{1\leq j\leq l} t_j(s)$. Then for $t> t_*$, the map 

\[ P_t = P^1_t\csor P^2_t\csor \dots \csor P^l_t\]

is the orthogonal projection from $L^2(G)$ to $\mathcal{H}(G)$ under the inner product

\[ \langle v, g \rangle_t =\int_G v(z) \overline{g(z)} \exp(-t\abs{z}^2),\]

which maps the space $\pss^s(G)$ to $\widetilde{\mathcal{H}}^s(G)= \widetilde{W}^s(G)\cap \mathcal{O}(G)$. Let $K_{t,z}$ be the Bergman kernel associated to $P_t$, i.e., 

\[ P_t \phi(z) = \langle \phi, K_{t,z}\rangle_t\]

for $\phi\in L^2(G)$. Let

\[ K^\alpha_{t,z}(w) = \frac{\partial}{\partial \overline{z}^\alpha} K_{t,z}(w).\]

Then 

\[ \frac{\partial}{\partial z^\alpha}(P_t \phi)(z) = \langle \phi,  K^\alpha_{t,z}\rangle_t.\]

Now if $s>\abs{\alpha}+n$, and $\phi\in\mathcal{C}^\infty_0(G)$, we have $P_t\phi\in \mathcal{H}^s(G)$,
so by  Sobolev embedding, the derivative $\frac{\partial^\alpha}{\partial z^\alpha}P_t \phi(z)$ is bounded independently of
$z$, if $\norm{\phi}_s$ is bounded. Hence
 
\begin{align*}
\norm{K^\alpha_{t,z}}_{-s} &= \sup_{\substack{\phi\in \mathcal{C}^\infty_0(G)\\\norm{\phi}_s=1}}\abs{\langle K^\alpha_{t,z},\phi\rangle_0}\\
&= \sup_{\substack{\phi\in \mathcal{C}^\infty_0(G)\\\norm{\phi}_s=1}}
\abs{\frac{\partial^\alpha}{\partial z^\alpha}P_t\left(e^{t\abs{z}^2}\phi\right)(z)}
\end{align*} 

is bounded independently of $z$. We now write 

\[ \exp(-t \abs{w}^2) = \sum_{\beta} c_\beta w^\beta \overline{w}^\beta\]

and note that

\begin{align*} \frac{\partial^\alpha}{\partial z^\alpha}H_r(z)&= \langle H_r, K^\alpha_{t,z}\rangle_t\\
&= \int_G H_r \overline{K^\alpha_{t,z}} \exp(-t \abs{w}^2)\\
&=\int_D \abs{u}^2 h^r \overline{K^\alpha_{t,z}\circ f} \exp(-t \abs{f}^2)\\
&=\sum_\beta c_\beta \int_D u f^\beta h^r\cdot \overline{u f^\beta (K^\alpha_{t,z}\circ f)}.\\
\end{align*}

For our fixed $s$, let $N$ be as in Lemma~\ref{lem-EN}. By using the first estimate in Lemma~\ref{lem-prodest} we have that

\begin{align*}\abs{\frac{\partial^\alpha}{\partial z^\alpha}H_r(z)}&=\sum_\beta\abs{c_\beta}\norm{u f^\beta h^r}_{kN}
\norm{u f^\beta \left(K^\alpha_{t,z}\circ f\right)}_{-N}\\
&\leq C \sum_{\beta}\abs{c_\beta} \norm{u f^\beta}_{kN} \times \sup_{z\in D} \abs{u f^\beta} \norm{K^\alpha_{t,z}}_{-s}\\
&\leq C \sum_{\beta} \abs{c_\beta} \norm{u f^\beta}_{kN}\times \sup_{z\in D} \abs{u f^\beta}\\
&\leq C\sum_{\beta} \abs{c_\beta} \norm{u f^\beta}_{kN+n}^2
\end{align*}

where we have used the facts that $h$ is smooth up to the boundary  (second line),  the estimate from Lemma~\ref{lem-EN} (second line),
the fact that  $\norm{K^\alpha_{t,z}}_{-s}$ is bounded as a function of $z$, and Sobolev embedding (last line).  Using the argument of 
Lemma~\ref{lem1}, we see that  $\norm{u f^\beta}_{kN+n}\leq C \norm{w^\beta}_{kN+n+Q}$ for a fixed $Q$ (independent of $\beta$.) 
Applying the Cauchy estimates to a ball of large radius $R$ containing the domain $G$, we see that 
$\norm{w^\beta}_{kN+n+Q}\leq C R^{\abs{\beta}}$. Therefore,
\begin{align*}\abs{\frac{\partial^\alpha}{\partial z^\alpha}H_r(z)}&\leq C \sum_{\beta} \abs{c_\beta} R^{2\beta}\\
&\leq C \exp\left(ntR^2\right)
\end{align*}
uniformly in $z\in G$. 
\end{proof}

\subsection{The proof continued}
As a consequence we have that

\begin{lem}
$u$ vanishes to finite order at every point of $\partial D$.
\end{lem}

\begin{proof}
The proof of this fact has been explained in \cite{Be2} and \cite{BC} for smoothly bounded domains. The argument for product domains or more generally Lipschitz domains is not different 
once we know that Proposition~\ref{prop-symfun} holds. Nevertheless, for the sake of completeness here are some details. Let $(F_k(w))_j$ denote the $j$--th component of the branch $F_k(w)$. By 
taking $h$ to be the coordinate functions in the above proposition, we see that the following pseudopolynomials 
\[
P_j(z, w) = \prod_{k = 1}^m \left( z_j - (F_k(w))_j  \right)
\]
which are monic in the variable $z_j$ have coefficients that are in $\mathcal O(G) \cap \mathcal C^{\infty}(\overline G)$. Note that $P_j(z, f(z)) = 0$ for all $1 \le j \le n$ since the 
graph of $f$ is an irreducible component of the variety defined by the vanishing of the $P_j$'s. Let $p \in \partial D$ be such that $u(p) = 
0$. Choose a sequence $\{p_l\} \in D$ converging to $p$ such that $u(p_l) \not= 0$ for all $l$. By passing to a subsequence we may assume that $f(p_l) \rightarrow p'\in \partial G$ and 
further that both $p, p'$ are the origins in $\mathbb C^n$. Thus $P_j(p_l, f(p_l)) = 0$ for all $1 \le j \le n$ and by letting $l \rightarrow \infty$ we get that $P_j(0, 0) = 0$. Since 
the coefficients of these pseudopolynomials are smooth up to $\partial G$, we may appeal to a quantitative version of the continuity of roots of monic polynomials (for example, see 
\cite{Ch} -- Chapter 1, Section 4) to conclude that for every $\epsilon > 0$, there exists a uniform constant $C > 0$ which is independent of $\epsilon$ such that if $\vert w \vert \le 
\epsilon^{m + 1}$ then $P_j(z, w) \not= 0$ for all $z$ with $\vert z_j \vert = C \epsilon$. 

\medskip

Let $\Delta(\epsilon)$ be the polydisc of polyradius $(\epsilon, \epsilon, \ldots, \epsilon)$ around $p = 0$ and let $B(\epsilon)$ be the ball of radius $\epsilon$ around $p'= 0$. 
We claim that $B(\epsilon^{m + 1}) \cap (G \setminus Z) \subset f(\Delta(C \epsilon) \cap D)$ for sufficiently small $\epsilon$, where $Z$ is as in the proof of Lemma 3.5. To show this, 
pick $w \in B(\epsilon^{m +1}) \cap (G \setminus Z)$ and let $\gamma : [0, 1] \rightarrow  B(\epsilon^{m +1}) \cap (G \setminus Z)$ be a path such that $\gamma(0) = f(p_{l_0})$ for some 
large fixed $l_0$ and $\gamma(1) = w$. This is possible since $f(p_l) \rightarrow p'= 0$. Let $F$ be a branch of $f^{-1}$ that is defined near $f(p_{l_0})$ and which maps it to 
$p_{l_0}$. Then $F$ admits analytic continuation along $\gamma$ and for each $t\in [0, 1]$ we have $P_j(F(\gamma(t)), \gamma(t)) = 0$ for $1 \le j \le n$. But it has been noted above 
that if $z \in \partial \Delta(C \epsilon)$ and $w \in B(\epsilon^{m + 1}) \cap G$, then at least one of the $P_j(z, w) \not= 0$. Therefore the continuous curve $F \circ \gamma$ 
cannot move out of $\Delta(C \epsilon) \cap D$. If we let $z = F(\gamma(1))$ then $z \in \Delta(C \epsilon) \cap D$. This means that at least one component of 
$f^{-1}(B(\epsilon^{m+1}) \cap (G \setminus Z))$ is contained in $\Delta(C \epsilon) \cap D$ which clearly implies the claim.

\medskip

To conclude, note that $\vert u \vert^2$ is the real Jacobian determinant of $f$ when viewed as a map from $\mathbb R^{2n}$ to itself and by the claim we see that
\[
\int_{\Delta(C \epsilon) \cap D} \vert u \vert^2 \ge {\rm Volume}(B(\epsilon^{m+1}) \cap G)
\]
as $Z$ has zero $2n$-dimensional volume. The integral can be dominated by $\epsilon^{2n}$ times the supremum of $\vert u \vert^2$ on $\Delta(C \epsilon) \cap D$ up to a uniform 
constant and the volume of $B(\epsilon^{m+1}) \cap G$ is greater than a uniform constant (which depends only on $G$) times $\epsilon^{2n(m+1)}$. Putting all this together, there is a 
constant $C'> 0$ independent of $\epsilon$ such that
\[
\sup_{\Delta(C \epsilon) \cap D} \vert u \vert \ge C'\epsilon^{mn}
\]
which shows that $u$ cannot vanish to infinite order at $p$.
\end{proof}

\noindent Next we show the following weaker version of conclusion (i) of Theorem~\ref{thm-main}:
\begin{lem}
The proper map $f:D\to G$  extends to a continuous map  from  $\overline{D}$ to  $\overline{G}$. 
\end{lem}

\begin{proof}
To show that $f$ admits a continuous extension to all points of $\partial D$, the weak division theorem from \cite{DF} can be applied here. To begin with, recall that 
$u \in \mathcal C^{\infty}(\overline D)$ and that $u \cdot (h \circ f) \in \mathcal C^{\infty}(\overline D)$ for all $h \in \mathcal H^{\infty}(G)$. Let $f = (f_1, f_2, \ldots, f_n)$ 
where $f_j \in \mathcal O(D)$ for $1 \le j \le n$. Taking $h = z_1^N$, $N \ge 1$, we see that $u \cdot f_1^N \in \mathcal H^{\infty}(D)$ 
for all $N \ge 1$. Pick an arbitrary point 
$p \in \partial D$ and let $L$ be a complex line that is transverse to the tangent cone to $\partial D$ at $p$ and which enters $D$ near $p$. The set of all such lines is open and 
non-empty. We may assume that $p=0$ and $L$ is  the $z_1$ axis $\{z_2=\dots=z_n=0\}$ 
in $\mathbb C^n$. Since $f_1$ is a bounded function on $D$, there exists a sequence $\{p_j\} \subset D 
\cap L$ such that $p_j \rightarrow 0$ and $f_1(p_j)$ converges; in fact after subtracting a constant from $f_1$ and still denoting the resulting function by $f_1$ we have that $f_1(p_j) 
\rightarrow 0$. For $g \in \mathcal O(D)$, let $g^{(s)} = \partial^s g/\partial z_1^s$.

\medskip

Suppose that $f_1$ does not extend continuously to the origin. Then there is a sequence $\{q_j\} \in D$ with $q_i \rightarrow 0$ such that $f_1(q_j) \rightarrow \gamma \not= 0$. That 
$u \cdot f_1^N$ is smooth on $\overline D$ for all $N \ge 1$ is used in the following way. Let $k < \infty$ denote the order of vanishing of $u$ at the origin and fix $N > k$. 
Since $u \cdot f_1^N \in \mathcal O(D) \cap \mathcal C^{\infty}(\overline D)$, it follows that its restriction to $L \cap D$ has an expansion around the origin of the form
\[
u \cdot f_1^N \vert L \cap D = \sum_{j = 0}^k b_j z_1^j + O(\vert z_1 \vert^{k + 1}).
\]
where $b_0 = b_1 = \ldots = b_{k - 1} = 0$ since $u$ is smooth on $\overline D$ and vanishes to order $k$ at the origin and $f_1^N$ is bounded. To see that $b_k = 0$, note that by 
induction, there exist universal polynomials $P_j$, $j \ge 1$ in $u, f_1$ and their derivatives such that 
\begin{equation}
(u \cdot f_1^M)^{(l)} = u^{(l)} \cdot f_1^M + \left( \sum_{j = 1}^l M^j P_j \right) \cdot f_1^{M - l}
\end{equation}
for all $l \ge 1$ and $M > l$. Now let $M = N$ and $l = k$. Along the sequence $\{p_j\}$, we see that the left side 
\[
(u \cdot f_1^N)^{(k)}(p_j) \rightarrow (u \cdot f_1^N)^{(k)}(0) = b_k
\]
while the right side contains $f_1^N$ and $f_1^{N - k}$ (note that $N >k$ by choice!) both of which converge to zero by assumption. Thus $b_k = 0$ as well.

\medskip

However, $f_1(q_j) \rightarrow \gamma \not= 0$ by assumption and by factoring out $f_1^{N - k}$ in (5.2) we see that
\begin{equation}
u^{(k)} \cdot f_1^k + \left(\sum_{j=1}^k N^j P_j\right) \rightarrow 0
\end{equation}
along $\{q_j\}$. Note that (5.3) holds for all $N > k$. Writing (5.3) for $2N$ we have
\begin{equation}
u^{(k)} \cdot f_1^k + \left(\sum_{j=1}^k (2N)^j P_j\right) \rightarrow 0
\end{equation}
along $\{q_j\}$. Multiplying the above equation by $2^{-k}$ and subtracting it from (5.3) gives
\begin{equation}
u^{(k)} \cdot f_1^k + \left(\sum_{j=1}^{k-1} N^j P_j (1 - 2^j 2^{-k})/(1 - 2^{-k})\right) \rightarrow 0
\end{equation}
along $\{q_j\}$. Note that (5.5) again holds for all $N > k$ and more importantly, the sum which involves $N^j P_j$ and other unimportant constants, now has only $k - 1$ terms. Repeat 
the above procedure -- write (5.5) for $2N$, multiply it by $2^{-k+1}$ and subtract it from (5.5). The first term in the resulting equation is still $u^{(k)} \cdot f_1^k$ while the 
sum now has only $k - 2$ terms. Proceeding this way, we finally get that
\[
u^{(k)}  \cdot f_1^k \rightarrow 0
\]
along $\{q_j\}$. This is a contradiction, since $f_1(q_j) \rightarrow \gamma \not= 0$ and $u^{(k)}(0) \not= 0$ since $u$ vanishes to order $k$. Thus $f_1$ and likewise all the other 
components of $f$ extend continuously to $p \in \partial D$.
\end{proof}

\noindent Since $f$ is proper and also continuous by the above result, we have $f(\partial D)= \partial G$. Denote the map $f:D\to G$ by
\[ 
f =(f_1, f_2, \dots,f_l)
\]
where $f_j$ is holomorphic from $D$ to $G^j$, and extends continuously to $\overline{D}$. We will need the following result of
Ligocka (\cite[Theorem~1]{L}:)
\begin{lem} For each $j$, with $1\leq j\leq l$, there is an $i$, with $1\leq i\leq k$, such that $f_j:D\to G^j$ depends only on on the 
factor $D^i$ in the product representation $D=D^1\times\dots \times D^k$. 
\end{lem}
\begin{proof} Assume without loss of generality that $j=1$. Let $V\subset \partial G^1$ be the set of strongly pseudoconvex points 
of the boundary of $G^1$. It is well-known that $V$ is non-empty and open, and when $G^1$ is a domain in the complex plane $\cx$, 
we take $V$ to be the whole of $\partial G^1$. Then $f_1^{-1}(V)$ is an open subset of $\partial D$, and therefore has a a non-empty intersection
with the smooth part $\partial D^{\rm reg}$ (since the latter is dense in $\partial D$.)
The smooth part   $\partial D^{\rm reg}$  itself is the disjoint union of $k$ pieces, each of which is the product of $k-1$ factors
$D^i$ with the boundary of the remaining factor $D^j$. Therefore, after renaming the indices, we can assume that 
\[ f_1^{-1}(V)\cap \left(\partial D^1 \times D^2 \times \ldots \times D^k\right) = f_1^{-1}(V)\cap\left( \partial D^1\times D'\right)\]
is non-empty and open. Here we are denoting the product $D^2 \times D^3 \times \ldots \times D^k$ by $D'$.  
Therefore, there is an open $U\subset \partial D^1$ and an open $W\subset D'$, such that $f_1$  maps the product
$U\times W$ into $V$. Fixing $u\in U$, this means that  the mapping $z\mapsto f_1(u,z)$  maps an open subset of  the complex variety $D'$
into the set $V$, which being strongly pseudoconvex does not contain any analytic varieties of positive dimension. It follows that $z\mapsto f_1(u,z)$
is locally constant on $W\subset D'$.  We now claim that the mapping $f_1:D^1\times D'\to G^1$ depends only on $D^1$.  For $z_1, z_2\in D'$, consider 
the function $g(w) = f_1(w,z_1)- f_1(w,z_2)$. Then $g$ is holomorphic on $D^1$ and continuous up to the boundary, and vanishes on the open 
subset $U$ of the boundary $\partial D^1$. It follows from the identity principle that $g\equiv 0$, so that $f_1$ is independent of $D'$.
\end{proof}

For each $i$, with $1\leq i \leq k$, let $J(i)$ denote the subset of $\{1,\dots, l\}$ such that for $\alpha\in J(i)$,
the map $f_\alpha:D\to G^\alpha$ depends only on $D^i$. By the above lemma, and the fact that $f$ is proper, it follows 
that each $J(i)$ is non-empty, the $J(i)$'s are disjoint and the union of all the $J(i)$'s is the set $\{1,\dots, l\}$.
If we define for $1\leq i \leq k$ 
\[ \mathcal{G}_i = \prod_{\alpha\in J(i)} G^\alpha,\]
we can represent the map $f$ as
\[
f = f^1 \times f^2 \times \ldots \times f^k
\]
where for each $1 \le i \le k$, the map 
\[
f^i : D^i \rightarrow \mathcal{G}_i
\]
is proper, and extends continuously to $\overline{D^i}$. Hence 
\begin{equation}
\mu_i \le \sum_{\alpha \in J(i)} \nu_{\alpha}
\end{equation}
for all $1 \le i \le k$. Adding these inequalities gives
\begin{equation}
n = \mu_1 + \mu_2 \ldots + \mu_k \le \sum_{i =1}^k \sum_{\alpha \in J(i)} \nu_{\alpha} \le n.
\end{equation}
This shows that each inequality in (5.6) must actually be an equality and thus we may conclude that
\[
f^i : D^i \rightarrow  \mathcal{G}_i
\]
is a proper holomorphic mapping between equidimensional domains for each $1\le i \le k$ that extends smoothly to $\partial D^i$.
Let 
\[ v_i = \det \left(f^i\right)',\]
be the complex Jacobian determinant of this mapping.  Taking $h\equiv 1$ in Lemma~\ref{lem1}, we conclude that $v_i\in \mathcal{H}^\infty(D_i)$. Further,
taking $h$ to be equal to the coordinate functions in Lemma~\ref{lem1},   we conclude that $f^i$ extends smoothly to 
$\overline{D^i}\setminus \{v_i=0\}$. Of course, from the previous work, we know that $f^i$ extends continuously to $\overline{D^i}$.

\medskip

Note that the closed set $\partial D^i \cap \{v_i=0\}$ is nowhere dense in $\partial D^i$, since if there is an open subset of  $D^i$  on which 
$v_i$ vanishes,  the uniqueness 
theorem will force $v_i$ to be identically zero in $D^i$ thus contradicting that $f^i$ is proper.  Now, the boundary $\partial D^i$ is 
strongly pseudoconvex on an open subset. Therefore, we can pick a point $p\in \partial D^i$, such that $\partial D^i$ is strongly pseudoconvex 
at $p$, and $v_i(p)\not=0$.   Then $f^i$ must map a small relatively open neighborhood of $p$ on $\partial D^1$ diffeomorphically into 
the smooth part of the boundary of $\mathcal{G}_i$. Hence $f^i(p)$ must be a strongly pseudoconvex point, but this is possible only if $\mathcal{G}_i$
consists of exactly one smooth factor. Therefore, each $J(i)$ consists of exactly one element  for all $1 \le i \le k$. 
It follows that the product mapping $f$ is such that each factor
\[
f^i : D^i \rightarrow G^j
\]
is a proper mapping between  smoothly bounded 
equidimensional domains. It now follows that $l=k$, and  setting $j= \sigma(i)$,  we obtain the permutation $\sigma$ of 
conclusion (iii). Clearly, we have a product representation
\[ f = f^1\times f^2\times \ldots \times f^k.\]
Further, by \cite{BC}, each $f^j$ extends as a smooth map from $\overline{D^j}$ to $\overline{G^{\sigma(j)}}$. It follows now that $f$ extends as a smooth map 
from $\overline{D}$ to $\overline{G}$ and the proof is complete.


\end{document}